\documentclass{amsart}
\usepackage{amsmath,amsthm,amssymb,xypic}
\usepackage[all]{xy}
\theoremstyle{plain}

\newtheorem{theorem}{Theorem}[section]

\newtheorem{proposition}[theorem]{Proposition}
\newtheorem{lemma}[theorem]{Lemma}
\newtheorem{corollary}[theorem]{Corollary}

\theoremstyle{definition}

\newtheorem{definition}[theorem]{Definition}

\theoremstyle{remark}

\newtheorem{claim}{Claim}

\newtheorem{case}[theorem]{Case}

\newtheorem{remark}[theorem]{Remark}
\DeclareMathOperator{\Sing}{Sing}

\newcommand{\QED}{\ifhmode\unskip\nobreak\fi\quad {\rm Q.E.D.}} % QED

\newcommand\iso{\cong}

\newcommand{\f}{\varphi}

\renewcommand{\L}{\mathcal{L}}

\renewcommand{\O}{\mathcal{O}}

\renewcommand{\P}{\mathbb{P}}

\newcommand{\Q}{\mathbb{Q}}

\newcommand{\rat}{\dasharrow}

\title{The minimal Cremona degree of quartic surfaces}

\author{Massimiliano Mella}
\address{Dipartimento di Matematica e Informatica\\
Universit\`a di Ferrara\\
Via Machiavelli 30\\
44121 Ferrara, Italia} \email{mll@unife.it}

\date{May 2021}

\begin{document}
\maketitle
{\hfill\it Dedicated to Ciro Ciliberto with admiration}
\vspace{1cm}
\begin{abstract}
  Two birational projective varieties in $\P^n$ are Cremona Equivalent if there is a
  birational modification of $\P^n$ mapping one onto the other. The
  minimal Cremona degree of $X\subset\P^n$ is the minimal integer
  among all degrees of varieties that are Cremona Equivalent to $X$.
  The Cremona Equivalence and the minimal Cremona degree
  is well understood for subvarieties of codimension at least $2$
  while both are in general very 
  subtle questions for divisors. In this note I compute the minimal
  Cremona degree of quartic surfaces in $\P^3$. This allows me to show
  that any quartic surface of elliptic ruled type has non trivial stabilizers in the
  Cremona group.
\end{abstract}

\section*{Introduction}
Birational geometry and birational maps are one of the most peculiar
aspects of Algebraic Geometry. Among the many interests of Ciro in all
realms of Algebraic Geometry, and actually Mathematics, an important
spot has to be reserved for birational arguments, not only for their
intrinsic interests but also for their link to the Italian school of
geometry and to projective geometry.

The most studied birational object is certainly the Cremona Group,
that is the group of birational self-maps of the projective space
$$Cr_n:=\{ f:\P^n\dasharrow\P^n | \mbox{ birational map}\}.$$
This group is wild, from almost all points of view, see \cite{Ca} for
a nice introduction. In this note I will focus my attention on a 
problem that is related to the wildness of $Cr_n$, the so called
Cremona Equivalence.

Let $X,Y\subset\P^N$ be irreducible and reduced birational
varieties. The subvariety $X$ and $Y$ are said to be  Cremona Equivalent if there is a
birational modification $\f:\P^N\dasharrow\P^N$  such that $\f(X)=Y$
and $\f$ is an isomorphism of the generic point of $X$.
The Cremona Equivalence problem has an old history that I  will
resume in Section \ref{sec:history} and a quite recent evolution
thanks to the modern tools of birational geometry inherited from
Minimal Model Program and Sarkisov program.

As a matter of fact any pair of birational projective varieties is
Cremona equivalent as long as their codimension is at least two. This
is a quite surprising result proved in \cite{MP1} and improved in
\cite{CCMRZ}. Note that this forces $Cr_n$ to contain, as a set, all
groups of birational modifications of its subvarieties.

The divisorial case is quite intricate.
It is easy to give examples of non Cremona Equivalent divisors,
\cite{MP1}, but it is quite hard to understand the divisors
that are Cremona Equivalent to a given one. For instance rational
divisor in $\P^n$  Cremona
Equivalent to a hyperplane are only known for $\P^2$ and in a less
precise way $\P^3$.

A natural notion arising from Cremona Equivalence, is that of minimal Cremona degree. see Definition~\ref{def:mCd}.
The complete classification of minimal Cremona degree plane curves is
known , \cite{MP2} and  \cite{CC1}, and Ciro together with Alberto Calabri
completed also the classification of minimal Cremona degree linear
system of plane curves, \cite{CC1}.

Recently in a series of papers, \cite{MP2} \cite{Me2} \cite{Me3},  I tried to shed some light on the
surface case and here I present the classification of minimal
Cremona degree surfaces of degree at most $4$, see Theorem~\ref{thm:class_mCd4}.
This is done following the main ideas in \cite{Me3} and plugging in
the detailed description of singularities of non rational quartic
surfaces obtained in a series of papers by Umezu and Urabe, \cite{Um1}
\cite{Um2} \cite{Ur}. This is the real bottleneck of my methods: the
need of a complete understanding of the singularities of divisors I am
considering. This prevents me to extend this classification to
surfaces of higher degrees.

I want to finish the introduction thanking Ciro for all he taught me
during our long friendship and for all the nice moments we shared both
in life and in mathematics. I am in debt, more than you ever
thought.

\section{History and background}
\label{sec:history}

Let $C\subset\P^2$ be an irreducible and reduced  plane curve. It is natural to ask what is the
minimal degree of curves that are equivalent to $C$ via a Cremona
modification. This is a classical problem studied since the ${\rm XIX^{th}}$
century by Cremona and Noether. More generally one can introduce the
notion of minimal Cremona degree as follows.

\begin{definition}\label{def:mCd}
  Let $X\subset\P^{n} $ be an irreducible and reduced
  hypersurface. The minimal Cremona degree of $X$ is
  $$\min\{d|\mbox{$X$ is Cremona Equivalent to a hypersurface of
    degree $d$}\}. $$
 The divisor $X$ is of minimal Cremona degree if its degree is equal
 to the minimal Cremona degree. That is it not possible to lower its
 degree with a Cremona modification in $Cr_n$.
\end{definition}

As I said the  case of plane curves has been widely treated  in
the old times, \cite{Ju} \cite{Ca1} \cite{Ca2}; see also the beautiful books of Coolidge
\cite{Co}and Conforto \cite{Con} for a complete account of the result
proved by that time. More recently the subject has been studied with the theory
of log pairs, \cite{Na}, \cite{Ii}, \cite{KM} and finally with a
mixture of old and new techniques a complete classification of minimal
Cremona degree irreducible plane curves has been achieved in 
\cite{CC1}, and \cite{MP2}.

As a matter of fact the Cremona equivalence
of a plane curve is dictated by its singularities but, unfortunately, its minimal
Cremona degree cannot be
guessed without a partial resolution of those, \cite[Example
3.18]{MP2}.  Due to this it is quite hard even in the plane curve case to
determine the minimal Cremona degree  of a fixed curve simply by its
equation.
The main tool developed in the ${\rm XX^{th}}$ century and improved by Ciro
and Alberto  is 
the theory of adjoint linear systems.
Let $D\subset\P^N$ be a divisor and $f:X\to\P^N$ a log
resolution of $(\P^N,D)$, with $D_X$ the strict transform
divisor. The adjoint linear system, with $m\geq n$, is
$$ adj_{n,m}(D)=f_*(|nD_X+mK_X|). $$
S. Kantor first noticed that the dimension of adjoint linear systems
is invariant under Cremona modifications.
It is easy to see that $adj_{n,m}(D)$ is independent of the resolution
$f$, as long as $m\geq n$, therefore a divisor of minimal Cremona
degree $1$ has all adjoint linear systems empty.

It is  quite natural to ask weather the opposite is true and actually
for plane curves this is one of the main results obtained at the
beginning of ${\rm XX^{th}}$ century.
\begin{theorem}[\cite{CE}] An irreducible and reduced curve $C\subset\P^2$ is
  Cremona Equivalent to a line if and only if all of its adjoint vanish.
\end{theorem}

In modern terms, also related to the Abhyankar--Moh
problem \cite{AM}, we may rephrase this result saying that a plane curve
$C$ is Cremona Equivalent to a line if and only if its log Kodaira
dimension is negative.
Pushing the theory of adjoint linear systems Ciro and Alberto 
were able to classify minimal Cremona degree curves,  minimal
Cremona degree linear systems
and the contractibility of configurations of lines,
\cite{CC1} \cite{CC2}.

It is then natural to investigate surfaces in $\P^3$, keeping in mind
that, quite often, the numerical invariants related to the canonical class and
its log variants are not subtle enough in higher dimensions. Think of  the beautiful 
Castelnuovo's rationality Theorem and the wild rationality behavior  of
Fano 3-folds.

For the Cremona Equivalence of surfaces it is useful to adopt  the
$\sharp$-Minimal Model Program, developed in \cite{Me1} or minimal model program
with scaling \cite{BCHMc}. In this way a  criterion for detecting
surfaces Cremona equivalent to a plane has been given in \cite{MP2}. The criterion, inspired by
the previous work of Coolidge on curves Cremona equivalent to lines
\cite{Co}, allows to determine all rational surfaces
that are Cremona equivalent to a plane, \cite[Theorem 4.15]{MP2}. 
Unfortunately, worse than in the plane
curve case, the criterion requires not only the
resolution of singularities but also a control on different log
varieties attached to the pair $(\P^3,S)$.

Let us start to enter in the Cremona Equivalence problem for surfaces with
some notations and definitions.
\begin{definition}
Let $(T,H)$ be a $\Q$-factorial uniruled
   3-fold and $H$ an irreducible and
   reduced Weil divisor
  on $T$. Let
  $$\rho=\rho_{H}=\rho(T,H)=:\mbox{  \rm sup   }\{m\in \Q|H+mK_T
\mbox{
  \rm is an effective $\Q$-divisor  }\}\geq 0,$$
 be the (effective) threshold of the pair
 $(T,H)$.
\end{definition}
\begin{remark} The threshold is not a
birational invariant of  pairs and it
is not preserved by blowing up.
Consider
 a plane $H\subset\P^3$ and let
$Y\to \P^3$ be the blow up of a point in $H$ 
then $\rho(Y,H_Y)=0$, while
$\rho(\P^3,H)=1/4$. For future reference note that both are less than
one.
\end{remark}
In \cite{MP2}, to overcome this problem it was introduced the notion
of good models and of sup threshold.

\begin{definition}
Let $(Y,S_Y)$ be a 3-fold pair. The pair $(Y,S_Y)$ is a birational
model of the pair $(T,S)$ if there is a birational map $\f:T\rat Y$
such that  $\f$ is well defined on the generic point of
$S$ and $\f(S)=S_Y$. % I say that
% $(Y,S_Y)$ is a canonical model if $Y$ has
% $\Q$-factorial singularities  and the pair $(Y,S_Y)$ has canonical
% singularities. 
A good model, \cite{MP2}, is a pair  $(Y,S_Y)$ with $S_Y$ smooth
and $Y$ terminal and $\Q$-factorial.
\end{definition}

\begin{remark}
 Let $(T,S)$ be a pair, to produce a good model it is enough to
 consider a log resolution of $(T,S)$. Clearly there are infinitely
 many good models for any pair and running a directed MMP one can find
 the one that is more suitable for the needs of the moment.
\end{remark}

The threshold  allowed to
produce an equivalent condition to being Cremona Equivalent to a
plane, \cite[Theorem 4.15]{MP2}, but unfortunately it is  almost impossible to
check this condition on specific examples. More recently, \cite{Me3}, a numerical
trick allowed to simplify the criteria  and provided an effective test
for a large class of rational surfaces.

\begin{lemma}[\cite{Me3}]\label{lem:rhominorediuno} Let $(T,S)$ and $(T_1,S_1)$ be
  birational models of a pair. Assume that $(T,S)$ has canonical singularities. If
  $\rho(T,S)=a\geq 1$ then $\rho(T_1,S_1)\geq a$.
\end{lemma}

 As a direct consequence of Lemma~\ref{lem:rhominorediuno} one can reformulate the condition of being Cremona
 Equivalent to a plane as follows.

 \begin{corollary}[\cite{Me3}]\label{cor:rhominorediuno} A rational surface $S\subset\P^3$ is Cremona equivalent
   to a plane if and only if there is a good model $(T,S_T)$ of
   $(\P^3,S)$ with $0<\rho(T,S_T)<1$.
 \end{corollary}

There is a class of divisor that are always Cremona Equivalent to a hyperplane.
\begin{remark}
  \label{rem:monoids} Let $S\subset\P^3$ be a monoid, that is an
  irreducible and reduced surface of degree $d$ with a point, say $p$, of
  multiplicity $d-1$. Then $S=(x_3F_{d-1}+F_d=0)$, consider the linear
  system 
$$\L:=\{(F_{d-1}x_0=0), (F_{d-1}x_1=0), (F_{d-1}x_2=0), S\}. $$
Then $\f_\L:\P^3\rat\P^3$ is a birational modification and $\f_\L(S)$
is a plane. That is any monoid is Cremona Equivalent to a plane.
\end{remark}

As a warm up I apply Corollary~\ref{cor:rhominorediuno}  and
Remark~\ref{rem:monoids} to  determine the minimal Cremona degree of all surfaces of degree at
most 3.

\begin{proposition}
  Let $S\subset\P^3$ be an irreducible and reduced surface of
  degree at most 3 and $\sigma$ its minimal Cremona degree.
  Then $\sigma\in\{1,3\}$ and 
  \begin{itemize}
  \item[$\sigma=1$] if and only if $S$ is rational
    \item[$\sigma=3$] if and only if $S$ is not rational, i.e. $S$ is
      a cone over an elliptic curve.
  \end{itemize}
\end{proposition}
\begin{proof} 
  The statement is immediate in degree 2 by Remark~\ref{rem:monoids}. Let $S$ be a rational
  cubic. If $S$ is smooth then $(\P^3,S)$ is a good model with
  $\rho(\P^3,S)=3/4$, hence we conclude by
  Corollary~\ref{cor:rhominorediuno}.
  If $S$ has a double point, then it is a monoid and Remark~\ref{rem:monoids}
  allows to  conclude. If $S$ is a cone I conclude by \cite{Me2}.
\end{proof}

My aim is to improve this result determining the minimal Cremona
degree of quartic surfaces in $\P^3$.

The case of quartics is, as usual, more subtle due to their own
intrinsic  complexity.
Smooth quartic surfaces are the only smooth hypersurfaces
with automorphisms not coming from linear automorphisms of $\P^n$,
\cite{MM}. In a recent paper K. Oguiso produced examples of isomorphic
smooth
quartic surfaces that are not Cremona Equivalent, \cite{Og}. It is a long standing problem
to determine which quartic surfaces are stabilized by non trivial
subgroups in $Cr_3$, that is for which quartic surface $S\subset\P^3$ there is a
Cremona modification $\omega:\P^3\rat\P^3$ such that $\omega$ is not
an isomorphism and $\omega(S)=S$. The above problem has been studied
by Enriques \cite{En} and Fano \cite{Fa} and also by Sharpe and coauthors in a series of papers,
\cite{MS} and \cite{SS}, at the beginning of the ${\rm XX^{\rm th}}$
century. More recently Araujo-Corti-Massarenti continued the study of mildly
singular quartic surfaces admitting a non trivial  stabilizers in the Cremona
Group, \cite{ACM}, in the context of Calabi-Yau pairs preserving
symplectic forms.

On the other hand the singularities of quartic surfaces are completely
classified, \cite{De}, and there are a few hundreds of non isomorphic
rational  quartic surfaces, \cite{Je16}. This allowed,  quite
surprisingly, to prove the following result.

\begin{theorem}[\cite{Me3}] \label{th:quartic} Let $S\subset\P^3$ be a rational quartic
  surface then $S$ is Cremona Equivalent to a plane.
\end{theorem}

This shows that any rational quartic has a huge
stabilizer in the Cremona group disregarding the type of singularity it may
have. Indeed it is amazing that, even if there are hundreds of non
isomorphic families of
rational quartics
the Cremona group of $\P^3$ is playable enough to smooth any of them
to a plane. In the next section I 
determine the minimal Cremona degree of an arbitrary quartic surface,
adapting the techniques used to prove Theorem~\ref{th:quartic} to an
arbitrary quartic surface. 

\section{Minimal Cremona degree of quartics}

My aim is to study the minimal Cremona degree of an arbitrary
quartic. The main tool I use, beside the $\sharp$-Minimal Model techniques,  is the complete classification of
singularities of quartic surfaces, see \cite{Je16} \cite{De} and in
particular the detailed analysis of the singular locus given in 
\cite{Um1}\cite{Um2} \cite{Ur}.

Let us  start treating quartic cones.

\begin{lemma}
  \label{lem:cones} Let $S$ be a quartic cone in $\P^3$. Then $S$ is
  of minimal degree if and only if its sectional genus is at least $2$.
\end{lemma}
\begin{proof}A surface of degree less then $4$ is either rational or
  an elliptic cone.
  By \cite[Corollary 2.7]{Me2} two surface cones are Cremona
  Equivalent if and only if their hyperplane sections are
  birational.
\end{proof}

Next I study non normal quartics.
\begin{lemma}\label{lem:non_isolated_sing} Let $S\subset\P^3$ be a non
  normal quartic,  which is not
  a cone, then $S$ is not of minimal Cremona degree.
\end{lemma}
\begin{proof} If $S$ is rational I apply
  Theorem~\ref{th:quartic}. Then I  assume that $S$ is
  not rational. Let $Y\subset S$ be the singular locus of $S$. Then by
  the classification  of \cite{Ur} I have that either $Y$ is a pair
  of skew lines or $Y$ is a line and the general hyperplane section,
  say $H$, has an
  $A_3$ singularity in $H\cap Y$.
 
 Here I mimic part of the proof in \cite[Proposition 2.4]{Me3}
Let $L\subset Y$ be a line and 
$x\in S$ a general point. Consider the
linear system  $\Lambda$ of
quadrics through $L$ and $x$. Let $\f:\P^3\rat\P^5$ be the map
associated to the linear system $\Lambda$. I have
$\f(\P^3)=Z\iso\P^1\times\P^2$, embedded via the Segre map, and $\f(S)=\tilde{S}$ is a divisor of type
$(3,2)$ in $\P^1\times\P^2$. Note that divisors of type $(1,0)$ are
planes and divisors of type $(0,1)$ are quadrics, then I have
$\deg\tilde{S}=3+4=7$.
\begin{claim}
  The surface $\tilde{S}$ is singular along a smooth conic.
\end{claim}
\begin{proof}[Proof of the Claim]
If $Y=L\cup R$ is a pair of lines then $L\cap R=\emptyset$ and clearly
$\\f(R)$ is a smooth conic, singular for $\tilde{S}$.
Assume that $Y=L$ and let $\nu:T\to \P^3$ be the blow up of $L$, with
exceptional divisor $E$. Then $E\cong\P^1\times\P^1$ and since the
general hyperplane section of $S$ has a singular point of type $A_3$
I have that $\nu_*^{-1}(S)\cap E$ is a conic and $\nu_*^{-1}(S)$ is
singular along this conic. This is enough to conclude. 
\end{proof}

Let $y\in\Sing(\tilde{S})$ be  a general point and $\pi:\P^5\rat\P^4$ the
projection from $y$. Then $\pi_{|Z}$ is a birational map, $Y:=\pi(Z)\subset\P^4$
   is a quadric of rank $4$, and $S_Q:=\pi(\tilde{S})$ is a
 surface of degree $7-2=5$.
  \begin{claim}\label{cl:always_singular}
    The vertex of the quadric is a smooth point  of $S_Q$.
  \end{claim}
  \begin{proof}
The surface    $\tilde{S}$ is a divisor of type $(3,2)$ in $Z$ and
    it is singular in $y$. Let $l$ and $P$, respectively, be the line
    and the plane passing through $x$ in $Z$. The general choice of
    $x\in S$ yields $l\not\subset \tilde{S}$. The line $l$ is
    mapped to the vertex of the quadric and $\tilde{S}_{|l}=2x+p$ for some point
    $p$. This shows that $S_Q$ contains the vertex of the quadric and
    it is smooth there.
  \end{proof}
 
  The 3-fold $Q$ is a quadric cone and $S_Q$ is singular along a
  line. Let $z\in \Sing(S_Q)$ be a point.
  By the Claim~\ref{cl:always_singular} $z$ is not
the vertex of $Q$. Thus the projection from $z$ produces a birational
model of $(Q, S_Q)$ , say $(\P^3,Z)$, with  $Z$ a cubic
surface. Therefore $S$ is not of minimal degree.
\end{proof}

 \begin{remark} Incidentally note that
    Lemma~\ref{lem:non_isolated_sing} gives a different proof of
    \cite[Proposition 2.6]{Ur}, where it is proven that a non normal
    quartic birational to a ruled surface over a curve of genus $2$ is
    a cone.
  \end{remark}

 Finally I treat the case of normal quartics. Let us first
 recall the following well known result, \cite[Proposition 8]{Um1}.
 \begin{proposition}\label{pro:Umezu}
The minimal resolution of a normal quartic surface $S\subset\P^3$ is
one of the following:
\begin{itemize}
\item[i)] a K3 surface
\item[ii)]  a rational surface
\item[iii)] birationally equivalent to an elliptic ruled surface
  \item[iv)] a ruled surface of genus 3.
  \end{itemize}.
  \end{proposition}

  It is immediate that quartic surfaces in i) and iv) are of minimal
  Cremona degree, see the proof of Theorem~\ref{thm:class_mCd4} for
  the details. Theorem ~\ref{th:quartic} treats surfaces in
  ii). Then I am left to study surfaces in iii). That is surfaces of
  elliptic ruled type.
  The main tool I use for this type of surfaces is the
  detailed description of their singularities contained in
  \cite{Um2}. I summarize what I need in the following Theorem.
  \begin{theorem}{\rm (\cite[Corollary pg 134]{Um2})}
    \label{th:sing_iso_4} Let $S\subset\P^3$ be a normal quartic of
    elliptic ruled type. Then the set of  irreducible components of a minimal  resolution of its singular
    locus contains either two disjoint elliptic curves or an elliptic
    curve, say $E$,  and one rational curve  intersecting $E$.
    In particular a minimal resolution has always at least two
    irreducible components with at least one elliptic curve.
  \end{theorem}
  I am ready to complete the analysis.
  \begin{lemma}\label{lem:isolated_sings}
    Let $S\subset \P^3$ be a normal quartic of elliptic ruled type.
   Then $S$ is not of minimal Cremona degree. 
  \end{lemma}
  \begin{proof} The surface $S$ is not rational and not a cone. In
  particular $S$ has only singular points of
  multiplicity $2$ and  by hypothesis $S$ has at least an irrational point.
  By \cite{Je16} and \cite{De} classification, the  irrational
 singularities are of the following type, in brackets the corresponding equation of $S$:
 \begin{itemize}
 \item[(1)] a double point with an
   infinitely near double line

   [$x_0^2x_1^2+x_0x_1Q_2(x_2,x_3)+F_4(x_1,x_2,x_3)=0$], 
 \item[(2)]  a tachnode with an infinitely near
   double line

   [$x_0^2x_1^2+x_0(x_2^3+x_1Q_2(x_2,x_3))+F_4(x_1,x_2,x_3)=0$]. 
\end{itemize}

%As in the previous Lemma I mimic \cite[Proposition 2.4]{Me3}.
\noindent Let $S$ be a  quartic with a singular point of type $(a)$, $a\in\{1,2\}$, and  let $\Lambda_a\subset|\O(2)|$ be the linear
system of quadrics having multiplicity  $a+1$  on the
valuation associated to the double line. Then it is easy to check that
the map
$$\f_{\Lambda_a}:\P^3\rat
X_a\subset\P^{7-a}$$
is birational.

As observed in \cite[Proposition 2.4]{Me3} we have two cases:
\begin{itemize}
\item[($S_1$)] $X_1\subset\P^6$ is the cone over
  the Veronese surface
  \item[($S_2$)]  $X_2\subset\P^5$ is the cone over 
the cubic surface $C\subset\P^4$, where $C$ is the projection of the
Veronese surface, say $V$,  from a point $z\in V$.
\end{itemize}

The main point here is that in both cases I have
$S_a:=\f_{\Lambda_a}(S)\subset|\O_{\P^{7-a}}(2)|$, in particular $\deg
X_a=5-a$ and $\deg
S_a=10-2a$.
\begin{case}[$S_2$] Assume that $S$ has a point of type (2).
Then the pair $(\P^3,S)$ is birational to $(X_2,S_2)$.
The surface $S_2\subset X_2\subset\P^5$ has degree 6 and $X_2$
has degree $3$. Let $x\in S_2$ be a general point and
$\pi:\P^5\rat\P^4$ the projection from $x$. Then $\pi(X_2)=Q$ is a
quadric cone and $S_x:=\pi(S_2)$ is a surface of degree 5. Hence there
is a a cubic hypersurface $D\subset\P^4$ such that
$$D_{|Q}=S_x+H,$$ for some plane $H$. Let $y\in S_x$ be a general
point and $\pi_y:\P^4\rat\P^3$ the projection from $y$.
\begin{claim}
  $\tilde{S}:=\pi_y(S_x)$ is a quartic surface singular along a line.
\end{claim}
\begin{proof}
  The point $y$ is general therefore $\deg\tilde{S}=4$. The map
  $\pi_{y|Q}$ is birational and it contract  the embedded tangent cone
  ${\mathbb T}_yQ\cap Q=\Pi_1\cup\Pi_2$ to a pair of lines $l_1\cup l_2$. Up to reordering I
  may assume that $H\cap \Pi_1$ is the vertex of the cone.  Therefore $\tilde{S}\cap
  \Pi_1$ is a cubic passing through $y$. Hence $\tilde{S}$ has
  multiplicity 2 along $l_1$.
\end{proof}
In particular the surface  $S$ is  Cremona Equivalent to a non normal quartic. 
\end{case}

\begin{case}[$S_1$] Assume that $S$ has a point of type (1). Then
  $(\P^3, S)$ is birational to $(X_1,S_1)$.
  First I prove that $S_1$ is always singular and on the smooth locus
  of $X_1$.
  \begin{claim} $S_1$ is in the smooth locus of $X_1$  and $S_1$ is singular
  \end{claim}
  \begin{proof} I need to describe deeper the map
    $\f:=\f_{\Lambda_1}:\P^3\rat X_1$, following \cite[Proposition 2.4]{Me3}.

    Let $S\subset\P^3$ be
  the quartic, I
  may assume that there is an irrational singular point of type (1) in
  $p\equiv[1,0,0,0]\in S$ and the equation of $S$ is
$$(x_0^2x_1^2+x_0x_1Q+F_4=0)\subset\P^3.$$
Let
   $\epsilon:Y\to\P^3$ be the weighted
blow up of $p$,
  with weights $(2,1,1)$ on the
  coordinates $(x_1,x_2,x_3)$, and
  exceptional divisor $E\iso\P(1,1,2)$. Then I have:
  \begin{itemize}
  \item[-] $\epsilon^*(x_1=0)=H+2E$,
  $\epsilon_{|H}:H\to (x_1=0)$ is an
  ordinary blow up and $H_{|E}$ is a smooth rational curve;
\item[-]  $\epsilon^*(S)=S_Y+4E$,
 \item[-] $S_{Y|E}$ has at most two irreducible components and in this
   case both curves are rational,
  \item[-]  $S_{Y|H}$ is a union of four smooth disjoint
  rational curves.
  \end{itemize}
  In particular:
  \begin{itemize}
  \item[-] both
  $H$ and $S_Y$ are on the smooth locus of $E$ and hence on the smooth
  locus of $Y$;
  \item[-] The surface $H$ is ruled by, the strict transforms of, the
    lines in the plane $(x_1=0)$ passing through the point $p$.
  \item[-] $S_Y$ has not further singularities along $H$
    \item[-] by Theorem~\ref{th:sing_iso_4} the surface $S_Y$ is not a
      resolution of singularities of $S$. That is $S_Y$ is singular.
  \end{itemize}
  Let
  $l_Y$ be a general curve in the ruling and
  $\Lambda_Y=\epsilon^{-1}_*(\Lambda_1)$ the
  strict transform linear system. Then  $E\cdot l_Y=1$ and by a direct computation I have
  \begin{itemize}
  \item[-] $\Lambda_Y\cdot
    l_Y=(\epsilon^*(\O(2))-2E)\cdot l_Y=0$
    \item[-] $S_Y\cdot
    l_Y=(\epsilon^*(\O(4))-4E)\cdot l_Y=0$
  \item[-] $K_Y\cdot l_Y=(\epsilon^*(\O(-4))+3E)\cdot l_Y=-1$.
  \item[-] $H\cdot l_Y=(\epsilon^*(\O(1))-2E)\cdot l_Y=-1$.
  \end{itemize}
  Then $H$  can be blown down to a smooth rational
  curve with a birational map $\mu:Y\to
  X_1$ and by construction $S_Y=\mu^*S_1$.
  This shows that  the unique singularity of
  $X_1$ is the singular point in $E$ and the surface $S_1$ is singular.
\end{proof}

Let $x\in\Sing(S_1)$  be a singular
point. Set $\pi:\P^6\rat\P^5$ be the projection from $
x$. Then $\pi_{|X_1}:X_1\rat X_2$ is birational and $\pi(S_1)\in
|\O_{X_2}(2)|$. I am therefore back to case ($S_2$). This shows
that, also in this case, $(\P^3,S)$ is Cremona Equivalent to a non
normal quartic.
\end{case}
To conclude observe that $S$ is of elliptic ruled type and it is
birational to a non normal quartic, say $V$. If
$V$ is a cone I conclude by Lemma~\ref{lem:cones} If $V$ is not a cone I apply
Lemma~\ref{lem:non_isolated_sing}.
\end{proof}

I am ready to compute the minimal Cremona degree of quartic surfaces.
\begin{theorem}\label{thm:class_mCd4}
  Let $S\subset\P^3$ be a quartic surface and $\sigma$ its minimal Cremona
  degree.  Then $\sigma\in\{1,3,4\}$ and 
  \begin{itemize}
  \item[$\sigma=1$] if and only if $S$ is rational
    \item[$\sigma=3$] if and only if it is of elliptic ruled type,
      i.e. it is birational to a ruled surface over an elliptic curve
      \item[$\sigma=4$] in all other cases, i.e. $S$ has at most
        rational double points or it is a cone of sectional genus at
        least $2$.
  \end{itemize}
\end{theorem}
\begin{proof}
 If $S$ is a cone I conclude by Lemma~\ref{lem:cones}. If $S$ is
 rational I conclude by Theorem~\ref{th:quartic}. 
  By \cite[Lemma 3.1]{MP1} if $S$ is not of minimal degree the pair
  $(\P^3,S)$ has worse than canonical singularities. Therefore I may
  assume that $S$ is not rational, is not a cone and $(\P^3,S)$ has worse then
  canonical singularities. If $S$ is not normal,by 
  Lemma~\ref{lem:non_isolated_sing} it is not of minimal Cremona
  degree and being not rational it has $\sigma=3$ and it is of
  elliptic ruled type. If $S$ is normal, by
  Proposition~\ref{pro:Umezu}, it is of elliptic ruled type. By
  Lemma~\ref{lem:isolated_sings}, the surface $S$  is not of minimal
  degree and being not rational it has $\sigma=3$. On
  the other hand all surfaces of degree at most $2$ are rational and non
  rational surfaces of degree $3$ are elliptic cones. Therefore
  $\sigma=3$ if and only if $S$ is of elliptic ruled type.
\end{proof}

Thanks to the detailed classification of singularities I am able to
easily characterize the minimal Cremona degree of quartic surfaces with
isolated  singularities.
\begin{corollary}
  \label{cor:sigma_sing} Let $S\subset\P^3$ be a quartic surface with
  isolated singularities and
  $\sigma$ its minimal degree.
  Then
  \begin{itemize}
  \item[$\sigma=1$] if and only if there is a unique  elliptic singularity,
    \item[$\sigma=3$] if and only if there are either two
  elliptic singularities or one singular point of genus $2$,
  \item[ $\sigma=4$] if and only if 
  it is a cone or has only rational double points.
  \end{itemize}
\end{corollary}
\begin{proof}
  Immediate by classification in \cite{De}.
\end{proof}
\begin{remark}
  A similar result for non isolated singularities is possible, thanks
  to \cite{Ur}, but it is not as neat as the one in
  Corollary~\ref{cor:sigma_sing}.
  
  It is hopeless to look for a similar statement in  higher
  degrees. The singularities of surfaces of degree greater than $4$ are not
  classified. Even the Cremona Equivalence of rational surfaces is not
  easy to tackle due to the lack of classification of  rational
  surfaces  of degree greater than $4$. The most intriguing problem is
  to determine whether the vanishing of adjoints is equivalent to the
  Cremona Equivalence to a plane, like in the plane curve case.
  % On the other hand there is a class of surfaces for which the
  % singularities are well known: the general projection of smooth
  % surfaces.
  % In the next section I will show how use this knowledge to determine
  % the Cremona Equivalence and minimal Cremona degree of some general projection.
\end{remark}

 Thanks to  Theorem~\ref{thm:class_mCd4} I am able to prove that any
  quartic surface whose minimal Cremona degree is less than $4$ has a
  non trivial stabilizer in the Cremona Group.
I start with the following probably known result that I prove for lack
of an adequate reference.
  \begin{lemma}
    \label{lem:cubic_cone_stab} Let $X\subset\P^{n}$ be a cubic hypersurface,
    then its stabilizer in $Cr_n$ is non trivial.
  \end{lemma}
  \begin{proof}
    Let $T\subset\P^{n+1}$ be a  cubic hypersurface with a double point in
    $[0,\ldots,0,1]$ and containing $X$
    has the hyperplane section $x_{n+1}=0$. If $(F_3=0)=X\subset\P^n$ it is enough to
    consider
    $$T=(x_{n+1}Q+F_3=0),$$
    for $Q\in{\mathbb C}[x_0,\ldots,x_n]$ general.
    Then the projection
    $$\pi:T\dasharrow \P^n=(x_{n+1}=0)\subset\P^{n+1}$$
    from the point $[0,\ldots,0,1]$ is
    a birational map such that $\pi(X)=X$.
    Fix a general point $p\in X$ and let $\tau_p:T\dasharrow T$ be the involution
    induced by $p$. That is $\tau_p(q)$ is the third point of
    intersection of the line spanned by $p$ and $q$.
    By construction $\tau(X)=X$. Hence $\pi\circ\tau\circ\pi^{-1}$ is a non
    trivial element in $Cr_n$ that stabilizes $X$.
  \end{proof}
  \begin{corollary}
    \label{pro:stab} Let $S\subset\P^3$ be a quartic surface of minimal
    Cremona degree different from $4$. Then $S$ has a non trivial
    stabilizer in $Cr_3$.
  \end{corollary}
  \begin{proof}
  Let $\sigma$ be the minimal Cremona degree of $S$, then
  $\sigma\in\{1,3\}$ by Theorem~\ref{thm:class_mCd4}. If $\sigma=1$
  the result is immediate. If $\sigma=3$ then $S$ is Cremona
  equivalent to a cubic cone and I apply
  Lemma~\ref{lem:cubic_cone_stab} to conclude.
\end{proof}
\begin{remark}
  Note that for quartic surfaces with minimal Cremona degree $4$ the
  situation is completely different and not much is known, see
  \cite{ACM} for a modern reference.
\end{remark}

\end{document}